\documentclass[preprint,review,12pt]{elsarticle}

\usepackage{amssymb}
\usepackage{hyperref}
\usepackage{amsthm}

\usepackage{amsmath}

\usepackage{graphicx}

\usepackage{float}

\usepackage{lineno}

\biboptions{numbers,sort&compress}

\journal{TBA}

\begin{document}

\begin{frontmatter}

\title{Analysis of new stabilized hp discontinuous Galerkin methods for elasticity problem\tnoteref{1}}\tnotetext[1]{The
work is supported by the Natural Science Foundation of China(No.
10901047). E-mail: zhihaoge@henu.edu.cn; tel:+86-13663786282, fax:+86-371-23881696.}
\author[a]{Zhihao Ge}
\address[a]{Institute of Applied Mathematics \& School of Mathematics and Statistics, Henan University, Kaifeng 475004, P.R. China}
\author[b]{Xiaogang Zhu}
\address[b]{School of Mathematics and Statistics, Henan University, Kaifeng 475004, P.R. China}

\begin{abstract}
In the paper, we propose three new hp discontinuous Galerkin methods for the elasticity problem and make a comparison of the three numerical methods. And we prove  the optimal order of convergence in energy norm and $L^2$-norm by the superpenalization technique. Finally, we give a numerical example to verify our theoretical results.
\end{abstract}


\end{frontmatter}

%%\linenumbers

\thispagestyle{empty}

%\vskip 0.5cm \baselineskip 16pt
\newtheorem{thm}{Theorem}[section]
\newtheorem{lma}{Lemma}[section]
\newtheorem{rem}{Remarks}[section]
\newtheorem{assu}{Assumption}[section]
\renewcommand{\theequation}{\arabic{section}.\arabic{equation}}

\section{Introduction}
\setcounter{equation}{0}

Elasticity problem is an important branch of solid mechanics, which describes  the changes of stress, strain, and displacement of the elastic medium by external factors. It is also the foundation of material mechanics, plastic mechanics and some interdisciplinary. Since the elasticity problem is very complicated, so it is a challenge due to the huge computation (see \cite{re2}). Many researchers studied the finite difference methods, finite element methods and some discontinuous Galerkin (DG) methods. From the view of physics, the DG method  is very natural to handle with the elasticity  problem because DG method is locally conservative, stable, high order accurate and easily adaptive. The first DG method was introduced in \cite{[7]W.H.Reed} by Reed and Hill for the hyperbolic equations, and then many DG methods were designed, one can refer \cite{1}. The DG methods for linear elasticity was firstly studied by \cite{5}, afterwards,  a local discontinuous Galerkin (LDG) method for linear elasticity problem is developed in \cite{6, re4}.  Cai and Ye presented a kind of mixed discontinuous finite element method in \cite{7}. Besides, Rivi\`{e}re, Shaw and Wheeler introduced a standard DG scheme for linear elasticity in \cite{8}. Houston and Sch\"{o}tzau gave
an adaptive mixed DG method for nearly incompressible linear elasticity in \cite{re3}. However, the above mentioned methods mainly consider the order of error estimates depending on $h$ for the linear elasticity problem. In the work, we propose absolutely stable hp DG methods for the elasticity problem, which are different from the general DG methods, and we prove the optimal order of convergence in the energy norm and $L^2$-norm by the superpenalization technique.

The remaining parts of this paper are organized as follows. In Section 2, we introduce some notations and the model problem. In Section 3, we derive the new hp DG methods for the elasticity problem and prove the stability of the methods. In Section 4,  we prove  the optimal order of convergence of our methods for the elasticity problem in the energy norm and $L^2$-norm. Finally, we give a numerical example to illustrate the performance of our theoretical results.

\section{Model problem and notations}
\setcounter{equation}{0}
In the paper, we consider the following elasticity problem:
\begin{eqnarray}\label{11}
  -\nabla\cdot \boldsymbol{\sigma} =& \boldsymbol{f}  & \quad \textrm{in $\Omega$}, \nonumber \\
  \boldsymbol{u} =& \boldsymbol{g}_D   & \quad \textrm{on $\mit\Gamma_D$}, \\
  \boldsymbol{\sigma}(\boldsymbol u)\boldsymbol{n} =& \boldsymbol{g}_N  & \quad \textrm{on $\mit\Gamma_N$}, \nonumber
\end{eqnarray}
where $\Omega\subset\mathbb{R}^n$ ( $n=2$ or $3$) is a convex polygonal domain with $\partial\Omega=\mit\Gamma_D\cup\mit\Gamma_N$, and the stress tensor $\boldsymbol{\sigma}(\boldsymbol u)=\lambda\nabla\cdot\boldsymbol{u}\mathcal{I}
  +2\mu\varepsilon(\boldsymbol{u})$, $\mathcal{I}$ is the identity tensor, $\sigma_{ij}=C_{ijkl}(\boldsymbol{x})\varepsilon_{kl}(\boldsymbol{u}) \quad \forall i, j, k, l=1, 2, \ldots, n$, $\boldsymbol{\varepsilon}(\boldsymbol{u})
  =\frac{1}{2}(\nabla\boldsymbol{u}
  +\nabla\boldsymbol{u}^T)$, and $C=(C_{ijkl}(\boldsymbol{x}))_{1\leq i,j,k,l\leq n}$ is a fourth-order tensor satisfying the symmetric property: $C_{ijkl}(\boldsymbol{x})=C_{jikl}(\boldsymbol{x}), C_{ijkl}(\boldsymbol{x})=C_{ijlk}(\boldsymbol{x}),
  C_{ijkl}(\boldsymbol{x})=C_{klij}(\boldsymbol{x})$. $\boldsymbol{f}$ is the external force, and $\boldsymbol{g}_D$ and $\boldsymbol{g}_N$ are the given functions. In the paper, we will omit the argument $\boldsymbol{x}$ in $\boldsymbol{C}$ and take the tensor $\boldsymbol{C}$ to be piecewise constant in $\Omega$.

 Let $\mathcal {T}_h$ be a nondegenerate quasiuniform subdivision of $\Omega$ with elements $K$. And we denote $h_K=
 \textrm{diam}(K)$, $h=\max\{h_K\}_{K\in \mathcal{T}_h}$, $\mit\Gamma=\bigcup\limits_{K\in\mathcal {T}_h}\partial K$ and $\mit\Gamma_{h}=\mit\Gamma \backslash\partial\Omega$, where $\partial K$ is the boundary of element $K$. Also, we let $e$ be the edge (face in 3D) of element $K$, and $\boldsymbol{n}$ be the unit outward vector normal to $\partial\Omega$.

To propose the numerical  methods, we need to introduce the following broken Sobolev spaces:
\begin{eqnarray*}
  H^s(\mathcal {T}_h) &=& \{v \in L^2(\Omega): v|_K \in H^s(K),\quad \forall K\in \mathcal {T}_h\}, \label{eq3}\\
  \boldsymbol{H}^s(\mathcal {T}_h) &=& \{\boldsymbol{v}\in (\boldsymbol{L}^2(\Omega))^n: v_i|_K \in H^s(K),\quad   1\leq i\leq n\}\label{eq4}.
\end{eqnarray*}
The norm associated with space $H^s(\mathcal {T}_h)$ is defined by
\begin{equation}\label{5}
  ||v||_{s,h}=\bigg(\sum_{K\in\mathcal {T}_h}||v||^{2}_{s,K}\bigg)^{1/2},
\end{equation}
where $||\cdot||_{s,K}$ is the usual Sobolev norm on element $K$.

The finite element space $V_h\subset\boldsymbol{H}^s(\mathcal {T}_h)$  is given by
\begin{equation}\label{6}
  V_h=\{\boldsymbol{v}: \boldsymbol{v}|_K \in\left(\mathbb{P}_r(K)\right)^n, \quad \forall K\in
  \mathcal {T}_h \},
\end{equation}
where $\mathbb{P}_r(K)$ is a space of polynomial of degree at most $r$ on $K$ for $r\geq1$.

Also, we introduce the average, jump operators and some approximation properties. For each interior
edge $e=\partial K^+\cap\partial K^-$ or boundary edge $e=\partial K^+\cap\partial\Omega$, we define
\begin{displaymath}
  \{\boldsymbol{v}\} := \left\{\begin{array}{ll}
  (\boldsymbol{v}^{+} + \boldsymbol{v}^{-} )/2 \quad  & \textrm{on} \quad \partial K \cap\mit\Gamma_{h},\\
  \boldsymbol{v}^+ \quad & \textrm{on} \quad \partial K\cap\partial\Omega,
  \end{array}\right.
\end{displaymath}
\begin{displaymath}
  [\boldsymbol{v}] := \left\{\begin{array}{ll}
  \boldsymbol{v}^{+} - \boldsymbol{v}^{-} \quad \quad & \textrm{on} \quad \partial K \cap\mit\Gamma_{h},\\
  \boldsymbol{v}^+ \quad & \textrm{on} \quad \partial K\cap\partial\Omega,
  \end{array}\right.
\end{displaymath}
where $\boldsymbol{v}^{\pm}(\boldsymbol{x})=\lim\limits_{\epsilon\rightarrow 0}\boldsymbol{v}(\boldsymbol{x}\pm\epsilon\boldsymbol{n})$.

It is well known that for $\phi\in H^s(K)$ there exists  $z^h_r\in \mathbb{P}_r(K)$  satisfying the
 following properties (cf. \cite{9}):
\begin{eqnarray}
  ||\phi-z^h_r||_{q,K} &\leq& C\frac{h_K^{\mu-q}}{r^{s-q}}||\phi||_{s,K}
  \quad s\geq 0,\label{7}\\
  ||\phi-z^h_r||_{0,e} &\leq& C\frac{h_K^{\mu-\frac{1}{2}}}{r^{s-\frac{1}{2}}}||\phi||_{s,K} \quad s > \frac{1}{2},\label{8}\\
  ||\phi-z^h_r||_{1,e} &\leq& C\frac{h_K^{\mu-\frac{3}{2}}}{r^{s-\frac{3}{2}}}||\phi||_{s,K} \quad s > \frac{3}{2}\label{9},
\end{eqnarray}
where $\mu=\min(r+1,s)$, $r=1,2,\ldots$ and $C$ is a constant depending on $s$ but independent of $\phi$, $h$, $r$.

Define the energy norm as follows:
\begin{equation}\label{nn22}
  |||\boldsymbol{v}|||^2=|||\boldsymbol{v}|||^2_{\mathcal{T}_h}+|||\boldsymbol{v}|||^2_{\partial\mathcal{T}_h},
\end{equation}
where
\begin{eqnarray*}
  |||\boldsymbol{v}|||^2_{\mathcal{T}_h}&=&\sum_{K\in\mathcal{T}_h}\int_K \boldsymbol{\sigma}(\boldsymbol{v}):\boldsymbol{\epsilon}(\boldsymbol{v})dx,\\
  |||\boldsymbol{v}|||^2_{\partial\mathcal{T}_h}&=&\frac{\beta r^2}{h}\sum_{e\in\mit\Gamma_{h}\cup\mit\Gamma_D}
  ||[\boldsymbol{v}]||^2_{0,e}+\frac{\gamma r^2}{h}\sum_{e\in\mit\Gamma_{h}\cup\mit\Gamma_D}
  ||[\boldsymbol{n}\cdot\boldsymbol{v}]||^2_{0,e}
\end{eqnarray*}
and $\beta$ and $\gamma$ are the stabilized parameters.

\section{Stabilized hp DG methods}
\setcounter{equation}{0}

Firstly, we give a variational problem of the problem (\ref{11}) as follows: Find $\boldsymbol{w}\in\boldsymbol{H}^s(\mathcal {T}_h)$ such that
\begin{equation}\label{14}
  B_h(\boldsymbol{w}, \boldsymbol{v})=L(\boldsymbol{v})  \quad  \forall \boldsymbol{v}\in \boldsymbol{H}^s(\mathcal {T}_h),
\end{equation}
where
{\setlength\arraycolsep{2pt}}
\begin{eqnarray}
   B_h(\boldsymbol{w},\boldsymbol{v}) & = & \sum_{K\in\mathcal{T}_h}\int_K\boldsymbol{\sigma}(\boldsymbol{w}):\boldsymbol{\epsilon}(\boldsymbol{v})dx-
  \sum_{e\in\mit\Gamma_{h}\cup\mit\Gamma_D}\int_e\{\boldsymbol{\sigma}(\boldsymbol{w})\boldsymbol{n}\}\cdot
  [\boldsymbol{v}]d\ell  \nonumber \\
   & +& \alpha\sum_{e\in\mit\Gamma_{h}\cup\mit\Gamma_D}\int_e\{\boldsymbol{\sigma}(\boldsymbol{v})\boldsymbol{n}\}
  \cdot[\boldsymbol{w}]d\ell
  +\frac{\beta r^2}{h}\sum_{e\in\mit\Gamma_{h}\cup\mit\Gamma_D}\int_e[\boldsymbol{w}]\cdot[\boldsymbol{v}]d\ell \nonumber \\
  & +&\frac{\gamma r^2}{h}\sum_{e\in\mit\Gamma_{h}\cup\mit\Gamma_D}\int_e[\boldsymbol{n}\cdot\boldsymbol{w}]
  [\boldsymbol{n}\cdot\boldsymbol{v}]d\ell,\label{15}
\end{eqnarray}
and
{\setlength\arraycolsep{2pt}}
\begin{eqnarray} \label{16}
  \boldsymbol{L}(\boldsymbol{v})& = & \int_{\Omega}\boldsymbol{f}\cdot\boldsymbol{v}dx + \int_{\mit\Gamma_N}\boldsymbol{g}_N\cdot\boldsymbol{v}d\ell + \alpha \int_{\mit\Gamma_D}
  \boldsymbol{\sigma}(\boldsymbol{v})\boldsymbol{n}\cdot\boldsymbol{g}_N d\ell  \nonumber \\
  & +&\frac{\beta r^2}{h}\int_{\mit\Gamma_D}\boldsymbol{g}_D\cdot\boldsymbol{v}d\ell
  + \frac{\gamma r^2}{h}\int_{\mit\Gamma_D}(\boldsymbol{n}\cdot\boldsymbol{w})
  (\boldsymbol{n}\cdot\boldsymbol{v})d\ell.
  \end{eqnarray}

As for the variational problem (\ref{14}), we have the following result:
\begin{thm}\label{lma1}
Let $s>\frac{3}{2}$. Suppose that the weak solution $\boldsymbol{u}$ of problem (\ref{11}) belongs to $\boldsymbol{H}^s(\mathcal {T}_h)$, then
$\boldsymbol{u}$ satisfies the variational formulation (\ref{14}). The converse is also valid if $\boldsymbol{u}$ belongs to $\boldsymbol{H}^1(\Omega)\cap\boldsymbol{H}^s(\mathcal {T}_h)$.
\end{thm}
\begin{proof}
Firstly, we prove that if the solution $\boldsymbol{u}$ of problem (\ref{11}) belongs to $\boldsymbol{H}^s(\Omega)$, then it solves (\ref{14}). To do this, multiplying the first equation of the problem (\ref{11}) by $\boldsymbol{v}\in V_h$ and integrating by parts, we get
\begin{equation}
\int_K\boldsymbol{\sigma}(\boldsymbol{u}):\boldsymbol{\epsilon}(\boldsymbol{v})dx-\int_{\partial
K}\boldsymbol{\sigma}(\boldsymbol{u})\boldsymbol{n}\cdot\boldsymbol{v}d\ell
=\int_K\boldsymbol{f}\cdot\boldsymbol{v}dx.\label{eq160120-1}
\end{equation}
Using (\ref{eq160120-1}), we have
\begin{equation}
  \sum_{K\in\mathcal{T}_h}\int_K\boldsymbol{\sigma}(\boldsymbol{u}):\boldsymbol{\epsilon}(\boldsymbol{v})dx-
  \sum_{e\in\mit\Gamma_{h}}\int_e\{\boldsymbol{\sigma}(\boldsymbol{u})\boldsymbol{n}\}\cdot[\boldsymbol{v}]d\ell=
  \int_{\Omega}\boldsymbol{f}\cdot\boldsymbol{v}dx+\int_{\partial\Omega}
  \boldsymbol{\sigma}(\boldsymbol{u})\boldsymbol{n}\cdot\boldsymbol{v}d\ell.\label{eq130825-1}
\end{equation}
Adding the term $\alpha\int_{\mit\Gamma_D}\boldsymbol{\sigma}(\boldsymbol{v})\boldsymbol{n}\cdot\boldsymbol{u}d\ell$
to both sides of (\ref{eq130825-1}) and using the boundary conditions, we have
\begin{eqnarray}
  \sum_{K\in\mathcal{T}_h}\int_K\boldsymbol{\sigma}(\boldsymbol{u}):\boldsymbol{\epsilon}(\boldsymbol{v})dx -
  \sum_{e\in\mit\Gamma_{h}\cup\mit\Gamma_D}\int_e\{\boldsymbol{\sigma}(\boldsymbol{u})\boldsymbol{n}\}\cdot
  [\boldsymbol{v}]d\ell +\alpha\int_{\mit\Gamma_D}\boldsymbol{\sigma}(\boldsymbol{v})\boldsymbol{n}\cdot
  \boldsymbol{u} d\ell   \\
  = \int_{\Omega}\boldsymbol{f}\cdot\boldsymbol{v}dx+\int_{\Gamma_N}\boldsymbol{g}_N\cdot\boldsymbol{v}d\ell+ \alpha\int_{\mit\Gamma_D}\boldsymbol{\sigma}(\boldsymbol{v})\boldsymbol{n}\cdot
  \boldsymbol{g}_D d\ell.\label{eq160120-2}
\end{eqnarray}
Using (\ref{eq160120-2}) and the fact of   $[\boldsymbol{u}]=0$, we see that (\ref{14}) holds.

Conversely, if $\boldsymbol{u}\in\boldsymbol{H}^1(\Omega)\cap\boldsymbol{H}^s(\mathcal {T}_h)$, then we have
\begin{eqnarray}
  &&\sum_{K\in\mathcal{T}_h}\int_K\boldsymbol{\sigma}(\boldsymbol{u}):\boldsymbol{\epsilon}(\boldsymbol{v})dx
  =-\sum_{K\in\mathcal{T}_h}\int_K \nabla\cdot\boldsymbol{\sigma}(\boldsymbol{u})\cdot\boldsymbol{v}dx \nonumber  \\
    &&
  +\sum_{e\in\mit\Gamma_{h}}\int_{e}\boldsymbol{\sigma}(\boldsymbol{u})\boldsymbol{n}\cdot[\boldsymbol{v}]d\ell
  +\int_{\partial\Omega}\boldsymbol{\sigma}(\boldsymbol{u})\boldsymbol{n}\cdot\boldsymbol{v}d\ell.\label{eq130825-2}
\end{eqnarray}
Using (\ref{eq130825-2}) and (\ref{14}), taking the suitable test functions $\boldsymbol{v}$, we know that $\boldsymbol{u}$ satisfies the problem (\ref{11}). The proof is completed
\end{proof}

Next, we propose the hp discontinuous Galerkin methods for the problem (\ref{11}): Find $\boldsymbol{u}_h\in V_h $ such that
\begin{equation}
  B_h(\boldsymbol{u}_h, \boldsymbol{v}_h)=L(\boldsymbol{v}_h)  \quad  \forall \boldsymbol{v}_h\in V_h,\label{15525-1}
\end{equation}
where $B_h(\cdot, \cdot)$ and $L(\cdot)$ are defined by (\ref{15}) and (\ref{16}), respectively.

\begin{rem}\label{rem1}
If the parameter $\alpha$ of (\ref{15}) is chosen to be $\{-1, 0, 1\}$, the methods are called symmetric interior penalty Galerkin (SIPG) method ($\alpha=-1$) and incomplete interior penalty Galerkin (IIPG) method ($\alpha=0$) and nonsymmetric interior penalty Galerkin (NIPG) method ($\alpha=1$),  respectively. We point out that the above methods are novel because the stabilized parameters are chosen dependently on the mesh size $h$ and the polynomial degree $r$, which are different from the general SIPG, IIPG and NIPG methods.
\end{rem}

\begin{rem}\label{rem2}
Denote $\Sigma_n(\boldsymbol{u}_h)$ by
\begin{displaymath}
\Sigma_n(\boldsymbol{u}_h)=\left\{
\begin{array}{ll}
\{\boldsymbol{\sigma}(\boldsymbol{u}_h)\boldsymbol{n}_K\}-r^2h^{-1}(\beta[\boldsymbol{u}_h]-\gamma [\boldsymbol{n}_K\cdot\boldsymbol{u}_h]\boldsymbol{n}_K)& \textrm{on $\mit\Gamma_{h}$}, \\
\boldsymbol{\sigma}(\boldsymbol{u}_h)\boldsymbol{n}_K- r^2h^{-1}(\beta(\boldsymbol{u}_h-\boldsymbol{g}_D)
-\gamma\boldsymbol{n}_K\cdot(\boldsymbol{u}_h-\boldsymbol{g}_D)\boldsymbol{n}_K)
&\textrm{on $\mit\Gamma_D$}, \\
\boldsymbol{g}_N& \textrm{on $\mit\Gamma_N$},
\end{array}\right.
\end{displaymath}
then we have
\begin{equation}
  \int_{\partial K}\Sigma_n(\boldsymbol{u}_h)+\int_K \boldsymbol{f}=\boldsymbol{0}\label{17}
\end{equation}
for all $K\in\mathcal{T}_h$.  That is, these schemes are local equilibrium in a weak sense for each element $K\in\mathcal{T}_h$.%, which is one advantage and an interesting property for DG methods compared with the classic finite element methods. In fact, it is easy to check that (\ref{17}) is right by choosing $v_i=1$ on $K$ with the other components zero and $\boldsymbol{v}=\boldsymbol{0}$ on $\Omega \backslash K$.
\end{rem}

%\begin{rem}\label{rem3}
%Multiplying (\ref{17}) by $\boldsymbol{\upsilon}\in KER(K)$ and move the term$(\boldsymbol{f},\boldsymbol{\upsilon})_K$ to the right side, we will find the flux normal $\Sigma_n(\boldsymbol{u}_h)$ satisfying
%\begin{equation}\label{18}
%  (\boldsymbol{\sigma}(\boldsymbol{u})\boldsymbol{n}_K,\boldsymbol{\upsilon})_{\partial K}=
 % -(\boldsymbol{f},\boldsymbol{\upsilon})_{K}=(\Sigma_n(\boldsymbol{u}_h),\boldsymbol{\upsilon})_{\partial K}\quad
%  \forall \boldsymbol{\upsilon}\in KER(K),
%\end{equation}
%where  $KER(K)$ is the kernel of strain tensor $\boldsymbol{\epsilon}(\cdot)$ and function $\boldsymbol{\upsilon}$ equal to 0 everwhere other than element $K$. This identity is also an important
%property in material science and it is actually equivalent with (\ref{17}) in some sense.
%\end{rem}

%\begin{rem}\label{rem4}
%Form Lemma \ref{lma1} and (\ref{15525-1}),  we have
%\begin{equation}\label{20}
 % B_h(\boldsymbol{u}-\boldsymbol{u}_h, % \boldsymbol{v})=\boldsymbol{0}  \quad \forall %\boldsymbol{v}\in V_h.
%\end{equation}
%\end{rem}

Next, we give the stability of our hp DG methods.

\begin{lma}\label{lma2}
For all $(\boldsymbol{w},\boldsymbol{v})\in V_h \times V_h$, then there exists a positive constant $C$ independent of $h$ and $r$ such that
\begin{equation}\label{21}
  B_h(\boldsymbol{w},\boldsymbol{v})\leq C_b|||\boldsymbol{w}||||||\boldsymbol{v}|||.
\end{equation}
\end{lma}
\begin{proof}
For all $(\boldsymbol{w},\boldsymbol{v})\in V_h \times V_h$, we have
\begin{eqnarray}
   B_h(\boldsymbol{w},\boldsymbol{v})  &= & \sum_{K\in\mathcal{T}_h}\int_K\boldsymbol{\sigma}(\boldsymbol{w}):\boldsymbol{\epsilon}(\boldsymbol{v})dx-
   \sum_{e\in\mit\Gamma_{h}\cup\mit\Gamma_D}\int_e\{\boldsymbol{\sigma}(\boldsymbol{w})\boldsymbol{n}\}\cdot
   [\boldsymbol{v}]d\ell \nonumber \\
   &+& \alpha\sum_{e\in\mit\Gamma_{h}\cup\mit\Gamma_D}\int_e\{\boldsymbol{\sigma}(\boldsymbol{v})
   \boldsymbol{n}\}\cdot[\boldsymbol{w}]d\ell+\frac{\beta r^2}{h}\sum_{e\in\mit\Gamma_{h}\cup\mit
  \Gamma_D}\int_e[\boldsymbol{w}]\cdot[\boldsymbol{v}]d\ell\nonumber \\
  &+&\frac{\gamma r^2}{h}\sum_{e\in\mit\Gamma_{h}\cup\mit\Gamma_D}\int_e[\boldsymbol{n}\cdot\boldsymbol{w}]\cdot
  [\boldsymbol{n}\cdot\boldsymbol{v}]d\ell \nonumber\\
  &=& T_1+ T_2+ T_3+ T_4+T_5.\label{eq130825-4}
\end{eqnarray}

For the term $|T_1|$, using the Cauchy-Schwarz inequality, we get
\begin{eqnarray}\label{22}
  |T_1|   &\leq& \sum_{K\in\mathcal{T}_h}
  \bigg(\int_KC_{ijkl}\epsilon_{kl}(\boldsymbol{w})\epsilon_{ij}(\boldsymbol{w})dx\bigg)^{\frac{1}{2}}
   \bigg(\int_KC_{ijkl}\epsilon_{kl}(\boldsymbol{v})\epsilon_{ij}(\boldsymbol{v})dx\bigg)^{\frac{1}{2}}\nonumber \\
  &\leq&|||\boldsymbol{w}||||||\boldsymbol{v}|||.
\end{eqnarray}
In order to bound the terms $|T_2|$ and $|T_3|$, we recall the following inverse estimate (cf. \cite{8, 10, 11})
\begin{equation}\label{23}
  ||\boldsymbol{C}^{1/2}\boldsymbol{\epsilon}(\boldsymbol{v})||_{0,e}\leq C_0 h^{-\frac{1}{2}}r
  ||\boldsymbol{C}^{1/2}\boldsymbol{\epsilon}(\boldsymbol{v})||_{0,K}  \quad \forall \boldsymbol{v}\in
  V_h,
\end{equation}
where $C_0$ is a positive constant independent of $h$ and $r$.

Assume $e\subset\partial K_1\cap\partial K_2$, by Cauchy-Schwarz and triangle inequality, we have
\begin{eqnarray}
  &&\int_e\{C_{ijkl}\epsilon_{kl}n_j\}[v_i]d\ell \leq ||\{C_{ijkl}\epsilon_{kl}n_j\}||_{0,e}||[v_i]||_{0,e} \nonumber\\
  && \leq \frac{1}{2}\left(||C_{ijkl}\epsilon_{kl}n_j|_{K_1}||_{0,e}+||C_{ijkl}\epsilon_{kl}n_j|_{K_2}||_{0,e}\right)||[v_i]||_{0,e} \nonumber\\
  && \leq C_0\rho h^{-1/2}r \left(||\boldsymbol{C}^{1/2}\boldsymbol{\epsilon}(\boldsymbol{v})||_{0,K_1\cup K_2}\right)||[\boldsymbol{v}]||_{0,e},
\end{eqnarray}
where $\rho$ is a positive constant with respect to $\boldsymbol{C}$. Furthermore, summing over all internal on edges, we have
\begin{eqnarray}
  |T_2| &\leq& C_0\rho\sqrt{\frac{n_0}{\beta}}\bigg(\sum_{K\in\mathcal{T}_h}||\boldsymbol{C}^{1/2}\boldsymbol{\epsilon}(\boldsymbol{v})
   ||^2_{0,K}\bigg)^{\frac{1}{2}}\bigg(\frac{\beta r^2}{h}\sum_{e\in\mit\Gamma_h\cup\mit\Gamma_D}||[\boldsymbol{v}]||^2_{0,e}\bigg)^{\frac{1}{2}}\nonumber\\
   &\leq& C_2|||\boldsymbol{w}||||||\boldsymbol{v}|||,
\end{eqnarray}
where the parameter $n_0$ is denoted by the maximum number of neighboring element.

As for the term $T_3$, taking the above argument, we get
\begin{equation}
  |T_3|\leq C_3 |||\boldsymbol{w}||||||\boldsymbol{v}|||.\label{eq130825-5}
\end{equation}

As for the terms $|T_4+T_5|$, we have
\begin{eqnarray}\label{25}
   |T_4+T_5|&\leq&\bigg|\frac{\beta r^2}{h}\sum_{e\in\mit\Gamma_{h}\cup\mit\Gamma_D}\int_e[\boldsymbol{w}]\cdot
   [\boldsymbol{v}]d\ell\bigg|+\bigg|\frac{\gamma r^2}{h}\sum_{e\in\mit\Gamma_{h}\cup\mit\Gamma_D}\int_e
   [\boldsymbol{n}\cdot\boldsymbol{w}][\boldsymbol{n}\cdot\boldsymbol{v}]d\ell\bigg|\nonumber \\
   &\leq& \bigg|\bigg(\frac{\beta r^2}{h}\sum_{e\in\mit\Gamma_{h}\cup\mit\Gamma_D}\int_e
   [\boldsymbol{w}]^2d\ell\bigg)^{\frac{1}{2}}\bigg(\frac{\beta r^2}{h}\sum_{e\in\mit\Gamma_{h}
   \cup\mit\Gamma_D}\int_e[\boldsymbol{v}]^2d \ell\bigg)^{\frac{1}{2}}\bigg|\nonumber \\
   &+& \bigg|\bigg(\frac{\gamma r^2}{h}\sum_{e\in\mit\Gamma_{h}\cup\mit\Gamma_D}\int_e[\boldsymbol{n}
   \cdot\boldsymbol{w}]^2d\ell\bigg)^{\frac{1}{2}}\bigg(\frac{\gamma r^2}{h}\sum_{e\in\mit\Gamma_{h}
   \cup\mit\Gamma_D}\int_e[\boldsymbol{n}\cdot\boldsymbol{v}]^2d\ell\bigg)^{\frac{1}{2}}\bigg|\nonumber \\
   &\leq& C_4|||\boldsymbol{w}||||||\boldsymbol{v}|||.
\end{eqnarray}

Combining with all the bounds together and taking $C=\max\{C_4, C_2, C_3\}$, we see that  (\ref{21}) holds. This completes the proof.
\end{proof}

\begin{lma}\label{lma3}
For all $\boldsymbol{w}\in V_h$, then there exists a positive constant $C_s$ independent of $h$ and $r$ such that
\begin{equation}\label{26}
  B_h(\boldsymbol{w},\boldsymbol{w})\geq C_s |||\boldsymbol{w}|||^{2}.
\end{equation}
\end{lma}
\begin{proof}
Using (\ref{15}) and letting  $\alpha=1$, we have
\begin{equation}
B_h(\boldsymbol{w},\boldsymbol{w})=|||\boldsymbol{w}|||, \forall \boldsymbol{w}\in V_h.\label{28}
\end{equation}
Using the Cauchy-Schwarz inequality, the inverse inequality and Young's inequality, we get
\begin{eqnarray}
&& \sum_{e\in\mit\Gamma_{h}\cup\mit\Gamma_D}\int_e\{C_{ijkl}
    \epsilon_{kl}(\boldsymbol{w})n_j\}[w_i]d\ell\nonumber\\
&&\leq\bigg(\frac{C^2_0n_0}{\beta}\sum_{K\in\mathcal{T}_h}|
|\boldsymbol{C}^{1/2}\boldsymbol{\epsilon}
  (\boldsymbol{w})||^{2}_{0,K}\bigg)^{\frac{1}{2}}\bigg(\frac{\beta r^2\rho^2}{h}
  \sum_{e\in\mit\Gamma_{h}\cup\mit\Gamma_D}
  ||[\boldsymbol{w}]||^2_{0,e}\bigg)^{\frac{1}{2}}\nonumber\\
 && \leq \frac{C^2_0 n_0 \eta}{2\beta}\sum_{K\in\mathcal{T}_h}\int_K\boldsymbol{\sigma}(\boldsymbol{w}):
  \boldsymbol{\epsilon}(\boldsymbol{w})dx
 +\frac{\beta r^2\rho^2}{2\eta h}
  \sum_{e\in\mit\Gamma_{h}\cup\mit\Gamma_D}||[\boldsymbol{w}]||^2_{0,e}.\label{eq130825-6}
\end{eqnarray}
Using (\ref{eq130825-6})  and (\ref{15}), we obtain
\begin{eqnarray}
  &&B_h(\boldsymbol{w},\boldsymbol{w})\geq \bigg(1-\frac{C^2_0 n_0\eta |1-\alpha|}{2\beta}\bigg)
  \sum_{K\in\mathcal{T}_h}\int_K\boldsymbol{\sigma}
  (\boldsymbol{w}):\boldsymbol{\epsilon}(\boldsymbol{w})dx+\\
 && \bigg(1-\frac{|1-\alpha|\rho^2}{2\eta}\bigg)\frac{\beta r^2}{h}
 \sum_{e\in\mit\Gamma_{h}\cup\mit\Gamma_D}||[\boldsymbol{w}]||^2_{0,e}+ \frac{\gamma r^2}{h}
 \sum_{e\in\mit\Gamma_{h}
 \cup\mit\Gamma_D}||
 [\boldsymbol{n}\cdot\boldsymbol{w}]||^2_{0,e}.\label{eq160119-1}
\end{eqnarray}
Choosing  $\eta$ such that $1-\frac{C^2_0 n_0\eta |1-\alpha|}{2\beta}>0$ and $1-\frac{|1-\alpha|\rho^2}{2\eta}>0$, and taking
$C_s=\min\left\{1-\frac{C^2_0 n_0\eta |1-\alpha|}{2\beta}, 1-\frac{|1-\alpha|\rho^2}
 {2\eta}\right\}$, using (\ref{28}) and  (\ref{eq160119-1}), we see that (\ref{26}) holds. The proof is completed.
\end{proof}

\begin{thm}\label{th1}
There is a unique solution $\boldsymbol{u}_h$ to the variational problem (\ref{14}).
\end{thm}
\begin{proof}
Suppose $\boldsymbol{u}^1_h$ and $\boldsymbol{u}^2_h$ are two
different solution of (\ref{14}), then we have
\begin{equation*}
  B_h(\boldsymbol{u}^1_h-\boldsymbol{u}^2_h,\boldsymbol{v})=\boldsymbol{0} \quad \forall
  \boldsymbol{v}\in V_h.
\end{equation*}
Choosing $\boldsymbol{v}=\boldsymbol{u}^1_h-\boldsymbol{u}^2_h$ and using Lemma \ref{lma3},
we have
\begin{equation*}
  |||\boldsymbol{u}^1_h-\boldsymbol{u}^2_h|||=0,
\end{equation*}
which implies that $\boldsymbol{u}^1_h=\boldsymbol{u}^2_h$.

Using Lemma \ref{lma2} and Lemma \ref{lma3}, we can easily prove the existence of the numerical solution $\boldsymbol{u}_h$
by Lax-Milgram theorem and Resize theorem for symmetric schemes and nonsymmetric schemes, respectively. And we omit the details of the proof. The proof is completed.
\end{proof}

\section{Error estimates}
\setcounter{equation}{0}
%\subsection{Error estimates in energy norm}

In this section, we will prove the optimal convergence rate in terms of $h$ and $r$ but suboptimal with respect to $r$ if $\boldsymbol{u}_I$ is discontinuous for all the above methods, where $\boldsymbol{u}_I$ is the interpolation of $\boldsymbol{u}$.

\begin{lma}\label{lma4}
Let $\boldsymbol{u}\in \boldsymbol{H}^2(\mathcal{T}_h)$ . If $\boldsymbol{u}_I\in C(\bar{\Omega})\cap V_h$, then we have
\begin{equation}\label{30}
  |B_h(\boldsymbol{u}-\boldsymbol{u}_I, \boldsymbol{v})|\leq C \frac{h^{\mu-1}}{r^{s-1}}
  ||\boldsymbol{u}||_s|||\boldsymbol{v}|||\quad \forall \boldsymbol{v}\in V_h;
\end{equation}
If $\boldsymbol{u}_I\notin C(\bar{\Omega})$, then we get
\begin{equation}\label{31}
 |B_h(\boldsymbol{u}-\boldsymbol{u}_I, \boldsymbol{v})|\leq C \frac{h^{\mu-1}}{r^{s-3/2}}
 ||\boldsymbol{u}||_s|||\boldsymbol{v}|||\quad \forall \boldsymbol{v}\in V_h,
\end{equation}
where $s\geq 2$, and $C$ is a positive constant independent of $h$ and $r$.
\end{lma}
\begin{proof}
Using (\ref{15}), we have
\begin{equation*}
    B_h(\boldsymbol{u}-\boldsymbol{u}_I,\boldsymbol{v}) = \sum_{K\in\mathcal{T}_h}\int_K\boldsymbol{\sigma}(\boldsymbol{u}-\boldsymbol{u}_I):\boldsymbol{\epsilon}
   (\boldsymbol{v})dx-
   \sum_{e\in\mit\Gamma_{h}\cup\mit\Gamma_D}\int_e\{\boldsymbol{\sigma}(\boldsymbol{u}
  -\boldsymbol{u}_I)\boldsymbol{n}\}\cdot[\boldsymbol{v}]d\ell
\end{equation*}
\begin{eqnarray}
  &+&\alpha\sum_{e\in\mit\Gamma_{h}\cup\mit\Gamma_D}\int_e\{\boldsymbol{\sigma}(\boldsymbol{v})\boldsymbol{n}\}
  \cdot[\boldsymbol{u}-\boldsymbol{u}_I]d\ell +
  \frac{\beta r^2}{h}\sum_{e\in\mit\Gamma_{h}\cup\mit\Gamma_D}
  \int_e[\boldsymbol{u}-\boldsymbol{u}_I]\cdot[\boldsymbol{v}]d\ell \nonumber \\
  &+&\frac{\gamma r^2}{h}\sum_{e\in\mit\Gamma_{h}\cup\mit\Gamma_D}\int_e[\boldsymbol{n}\cdot(\boldsymbol{u}
  -\boldsymbol{u}_I)]\cdot[\boldsymbol{n}\cdot\boldsymbol{v}]d\ell \nonumber \\
  &\leq& Q_1+Q_2+Q_3+Q_4+Q_5. \label{32}
\end{eqnarray}

As for every term of (\ref{32}), using the Cauchy-Schwarz inequality and (\ref{7})-(\ref{9}), we have
\begin{eqnarray}
 |Q_1|  &\leq & \sum_{K\in\mathcal{T}_h}\bigg(\int_K C_{ijkl}\epsilon_{kl}(\boldsymbol{u}-\boldsymbol{u}_I)
  \epsilon_{ij}(\boldsymbol{u}-\boldsymbol{u}_I)dx\bigg)^{\frac{1}{2}}\bigg(\int_KC_{ijkl}
  \epsilon_{kl}(\boldsymbol{v})\epsilon_{ij}(\boldsymbol{v})dx\bigg)^{\frac{1}{2}}\nonumber \\
&\leq & C \frac{h^{\mu-1}}{r^s}||\boldsymbol{u}||_s|||\boldsymbol{v}|||,\label{33}\\
  |Q_2| &\leq & \bigg(\frac{h}{\beta r^2}\sum_{e\in\mit\Gamma_{h}\cup\mit\Gamma_D}||\{C_{ijkl}\epsilon_{kl}
  (\boldsymbol{u}-\boldsymbol{u}_I)n_j\}||^2_{0,e}\bigg)^{\frac{1}{2}}
  \bigg(\frac{\beta r^2}{h}\sum_{e\in\mit\Gamma_{h}\cup\mit\Gamma_D}||[v_i]||^2_{0,e}
  \bigg)^{\frac{1}{2}}\nonumber \\
  &\leq& C\frac{h^{1/2}}{r}\bigg(\sum_{e\in\mit\Gamma_{h}\cup\mit\Gamma_D}||\boldsymbol{u}
  -\boldsymbol{u}_I||^2_{1,e}\bigg)^{1/2}|||\boldsymbol{v}|||\nonumber\\
  &\leq & C\frac{h^{\mu-1}}{r^{s-1/2}}||\boldsymbol{u}||_s|||\boldsymbol{v}|||
  \nonumber \\
  &\leq & C\frac{h^{\mu-1}}{r^{s-1}}||\boldsymbol{u}||_s|||\boldsymbol{v}|||,\label{34}
\end{eqnarray}
where we use the inequality $1\leq 1/r^{-\frac{1}{2}}$.

Since the terms $Q_3, Q_4$ and $Q_5$ vanish if the interpolation $\boldsymbol{u}_I\in C(\bar{\Omega})\cap V_h$ due to the jump $[\boldsymbol{u}-\boldsymbol{u}_I]=0$ on each edge, so we see that (\ref{30}) holds.

As for the case of $\boldsymbol{u}_I\notin C(\bar{\Omega})$, using the Cauchy-Schwarz inequality, the inverse estimate and (\ref{8}), we have
\begin{eqnarray}
|Q_3|&\leq& Ch^{-1/2}r\bigg(\sum_{K\in\mathcal{T}_h}||\boldsymbol{C}^{1/2}\boldsymbol{\epsilon}(\boldsymbol{v})||^2_{0,K}\bigg)^{\frac{1}{2}}
  \bigg(\sum_{e\in\mit\Gamma_h\cup\mit\Gamma_D}||[\boldsymbol{u}-\boldsymbol{u}_I]||^2_{0,e}\bigg)^{\frac{1}{2}}\nonumber\\
&\leq & C\frac{h^{\mu-1}}{r^{s-3/2}}||\boldsymbol{u}||_s|||\boldsymbol{v}|||,\label{36}\\
|Q_4+Q_5| &\leq&\sum_{e\in\mit\Gamma_{h}\cup\mit\Gamma_D}\bigg(\frac{\beta r^2}{h}\int_e[\boldsymbol{u}-
  \boldsymbol{u}_I]^2d\ell\bigg)^{\frac{1}{2}}\bigg(\frac{\beta r^2}{h}\int_e[\boldsymbol{v}]^2d\ell
  \bigg)^{\frac{1}{2}}  \nonumber \\
   &+& \sum_{e\in\mit\Gamma_{h}\cup\mit\Gamma_D} \bigg(\frac{\gamma r^2}{h}\int_e[\boldsymbol{n}\cdot
   (\boldsymbol{u}-\boldsymbol{u}_I)]^2d\ell\bigg)^{\frac{1}{2}}\bigg(\frac{\gamma r^2}{h}
  \int_e[\boldsymbol{n}\cdot\boldsymbol{v}]^2d\ell\bigg)^{\frac{1}{2}} \nonumber\\
   &\leq & Ch^{-1/2}r\bigg(\sum_{e\in\mit
   \Gamma_{h}\cup\mit\Gamma_D}||[\boldsymbol{u}- \boldsymbol{u}_I]||^2_{0,e}\bigg) ^{1/2}
   |||\boldsymbol{v}||| \nonumber\\
  &\leq & C\frac{h^{\mu-1}}{r^{s-3/2}}
   ||\boldsymbol{u}||_s|||\boldsymbol{v}|||.\label{eq160119-2}
\end{eqnarray}

Using (\ref{33}), (\ref{34}), (\ref{36}), (\ref{eq160119-2}), we see that (\ref{31}) holds. This completes the proof.
\end{proof}

Next, we give the main result as follows:
\begin{thm}\label{th2}
Under the assumption of Lemma \ref{lma4}, there is a positive constant independent of $h$ and $r$ such
that
\begin{equation}\label{39}
  |||\boldsymbol{u}-\boldsymbol{u}_h|||\leq C \frac{h^{\mu-1}}{r^{s-1}}||\boldsymbol{u}||_s.
\end{equation}
If the interpolation $\boldsymbol{u}_I\notin C(\bar{\Omega})$, then
\begin{equation}\label{40}
  |||\boldsymbol{u}-\boldsymbol{u}_h|||\leq C \frac{h^{\mu-1}}{r^{s-3/2}}||\boldsymbol{u}||_s
\end{equation}
holds for $s\geq 2$.
\end{thm}
\begin{proof}
Usinging (\ref{26}), we have
\begin{eqnarray}
  C_s|||\boldsymbol{u}-\boldsymbol{u}_I|||^2 &\leq& B_h(\boldsymbol{u}-\boldsymbol{u}_I,\boldsymbol{u}-\boldsymbol{u}_I) \nonumber \\
   &\leq& C \frac{h^{\mu-1}}{r^\tau}||\boldsymbol{u}||_s|||\boldsymbol{u}-\boldsymbol{u}_I|||, \label{41}\\
  C_s |||\boldsymbol{u}_h-\boldsymbol{u}_I|||^2&\leq& B_h(\boldsymbol{u}_h-\boldsymbol{u}_I,\boldsymbol{u}_h-\boldsymbol{u}_I) \nonumber \\
   &=& B_h(\boldsymbol{u}-\boldsymbol{u}_I,\boldsymbol{u}_h-\boldsymbol{u}_I)-
   B_h(\boldsymbol{u}-\boldsymbol{u}_h,\boldsymbol{u}_h-\boldsymbol{u}_I)\nonumber\\
   &=& B_h(\boldsymbol{u}-\boldsymbol{u}_I,\boldsymbol{u}_h-\boldsymbol{u}_I) \nonumber\\
   &\leq & C \frac{h^{\mu-1}}{r^\tau}||\boldsymbol{u}||_s|||\boldsymbol{u}_h-\boldsymbol{u}_I|||,\label{42}
\end{eqnarray}
where $\tau=s-1$ if $\boldsymbol{u}_I$ is continuous, otherwise $\tau=s-3/2$.
%Simplify the above inequalities to get
%\begin{eqnarray}
 % |||\boldsymbol{u}-\boldsymbol{u}_I||| &=& C \frac{h^{\mu-1}}{r^\tau}||\boldsymbol{u}||_s \label{43}\\
%  |||\boldsymbol{u}_h-\boldsymbol{u}_I||| &=&  C \frac{h^{\mu-1}}{r^\tau}||\boldsymbol{u}||_s.\label{44}
%\end{eqnarray}

Using the triangle inequality and (\ref{42}), we have
\begin{eqnarray}\label{45}
  |||\boldsymbol{u}-\boldsymbol{u}_h||| &\leq& |||\boldsymbol{u}-\boldsymbol{u}_I|||+
  |||\boldsymbol{u}_h-\boldsymbol{u}_I||| \nonumber\\
   &\leq&  C \frac{h^{\mu-1}}{r^\tau}||\boldsymbol{u}||_s,
\end{eqnarray}
which completes the proof.
\end{proof}

\begin{rem}\label{rem5}
From (\ref{39}), we know that the error estimate is optimal in terms of  both $h$-convergence and $r$-convergence, however, (\ref{40}) shows that the error estimate is optimal in terms of $h$-convergence but suboptimal with respect to the
polynomial degree $r$.
\end{rem}

Next, we prove the error estimates in $L^2$-norm. As for SIPG method, we easily achieve the optimal order
convergence in $L^2$-norm by Aubin-Nitsche technique because the method satisfies the following adjoint consistency condition
\begin{equation}\label{46}
  B_h(\boldsymbol{v},\boldsymbol{u})=\int_{\Omega}\boldsymbol{v}\cdot\boldsymbol{f}dx \quad
  \forall \boldsymbol{v}\in V_h.
\end{equation}
However, the argument fails for IIPG method and NIPG method which are adjoint inconsistent, so we display the superpenalizaion term and show that the optimal order convergence in $L^2$-norm, our main idea mainly comes from \cite{4} and \cite{12}. As for IIPG method and NIPG method, we choose the superpenalizaion terms as follows:
\begin{equation}\label{48}
  \mathcal{K}(\boldsymbol{w},\boldsymbol{v})=\frac{\beta r^2}{h^d}\sum_{e\in\mit\Gamma_{h}\cup
  \mit\Gamma_D}\int_e[\boldsymbol{w}]\cdot[\boldsymbol{v}]d\ell+\frac{\gamma r^2}{h^d}\sum_{e\in\mit
  \Gamma_{h}\cup\mit\Gamma_D}\int_e[\boldsymbol{n}\cdot\boldsymbol{w}][\boldsymbol{n}\cdot\boldsymbol{v}]d\ell.
\end{equation}
Define the new energy norm as
\begin{equation}\label{49}
  |||\boldsymbol{v}|||=\left(|||\boldsymbol{v}|||^2_{\mathcal{T}_h}+\mathcal{K}(\boldsymbol{v},
  \boldsymbol{v})\right)^{\frac{1}{2}}.
\end{equation}

It is easy to check that the boundedness, stability and Theorem (\ref{th2}) still hold with respect to the new energy norm (\ref{49}).

Now, we give the following main result:
\begin{thm}\label{th3}
For SIPG method, there exists a positive constant $C$ independent of $h$ such that
\begin{equation}\label{51}
  ||\boldsymbol{u}-\boldsymbol{u}_h||_{0}\leq C h^{\mu}||\boldsymbol{u}||_s.
\end{equation}
For IIPG method and NIPG method, the optimal error estimate
also can be achieved if $d\geq 3$ under the assumptions of (\ref{48}) and (\ref{49}).
\end{thm}
\begin{proof}
As for SIPG method, we consider the dual problem:
\begin{eqnarray}\label{53}
  -\nabla\cdot\boldsymbol{\sigma}(\boldsymbol{\varphi}) = \boldsymbol{u}-\boldsymbol{u}_h \quad
   \textrm{in $\Omega$},\quad \quad  \boldsymbol{\sigma}(\boldsymbol{\varphi})\boldsymbol{n}
   = \boldsymbol{0}\quad \textrm{on $\partial\Omega$}.
\end{eqnarray}
Taking $\boldsymbol{v}=\boldsymbol{u}-\boldsymbol{u}_h$, we have
\begin{eqnarray}
  ||\boldsymbol{u}-\boldsymbol{u}_h||^2_{0} &=& B_h(\boldsymbol{u}-\boldsymbol{u}_h,\boldsymbol{\varphi}) \nonumber\\
   &=& B_h(\boldsymbol{u}-\boldsymbol{u}_h,\boldsymbol{\varphi}-\boldsymbol{\varphi}_I)+
   B_h(\boldsymbol{u}-\boldsymbol{u}_h,\boldsymbol{\varphi}_I) \nonumber\\
   &=& B_h(\boldsymbol{u}-\boldsymbol{u}_h,\boldsymbol{\varphi}-\boldsymbol{\varphi}_I)\label{55}\\
  &\leq& C_b|||\boldsymbol{\varphi}-\boldsymbol{\varphi}_I|||
  |||\boldsymbol{u}-\boldsymbol{u}_h|||\nonumber\\
  &\leq& C h||\boldsymbol{\varphi}||_2|||\boldsymbol{u}-\boldsymbol{u}_h|||,\label{15526-1}
\end{eqnarray}
where $\boldsymbol{\varphi}_I$ is the interpolation of $\boldsymbol{\varphi}_I$.

Due to the elliptic regularity, we have
\begin{equation}\label{57}
  ||\boldsymbol{\varphi}||_2\leq C||\boldsymbol{u}-\boldsymbol{u}_h||_{0}.
\end{equation}
Using Theorem \ref{th2}  and (\ref{57}), we get
\begin{eqnarray}\label{58}
  ||\boldsymbol{u}-\boldsymbol{u}_h||_{0} \leq  Ch |||\boldsymbol{u}-\boldsymbol{u}_h|||
  \leq  Ch^{\mu}||\boldsymbol{u}||_s.
\end{eqnarray}

As for IIPG method and NIPG method, we have
\begin{eqnarray}\label{59}
 ||\boldsymbol{u}-\boldsymbol{u}_h||^2_{0}&=&B_h(\boldsymbol{u}-\boldsymbol{u}_h,\boldsymbol{\varphi})\nonumber\\
 &&-
 \theta\sum_{e\in\mit\Gamma_{h}\cup\mit\Gamma_D}\int_e\{\boldsymbol{\sigma}(\boldsymbol{\varphi})
  \boldsymbol{n}\}\cdot[\boldsymbol{u}-\boldsymbol{u}_h]d\ell.
  \end{eqnarray}

Using the Cauchy-Schwarz inequality and the inverse estimate, we obtain
\begin{eqnarray}\label{50}
  &&\sum_{e\in\mit\Gamma_{h}\cup\mit\Gamma_D}\int_e\{\boldsymbol{\sigma}(\boldsymbol{w})
  \boldsymbol{n}\}\cdot[\boldsymbol{v}]d\ell  \nonumber\\
  & & \leq \sum_{e\in\mit\Gamma_{h}\cup\mit\Gamma_D}\bigg(\frac{h^d}{\beta r^2}\int_e
  |\{C_{ijkl}\epsilon_{kl}(\boldsymbol{w})n_j\}|^2d\ell\bigg)^{\frac{1}{2}}\bigg(\frac{\beta r^2}
  {h^d}\int_e[v_i]^2d\ell\bigg)^{\frac{1}{2}} \nonumber\\
  &&\leq \bigg(\frac{h^d}{\beta r^2}\sum_{e\in\mit\Gamma_{h}\cup\mit\Gamma_D}
  ||\{C_{ijkl}\epsilon_{kl}(\boldsymbol{w})n_j\}||^2_{0,e}\bigg)^{\frac{1}{2}}
  \bigg(\frac{\beta r^2}{h^d}\sum_{e\in\mit\Gamma_{h}\cup\mit\Gamma_D}||[v_i]||
  ^2_{0,e}\bigg)^{\frac{1}{2}} \nonumber\\
  && \leq Ch^{\frac{d-1}{2}}\bigg(\sum_{K\in\mathcal{T}_h}||
  \boldsymbol{C}^{1/2}\boldsymbol{\epsilon}(\boldsymbol{w})||^2_{0,K}\bigg)^{\frac{1}{2}}
  \bigg(\frac{\beta r^2}{h^d}\sum_{e\in\mit\Gamma_{h}
  \cup\mit\Gamma_D}||[\boldsymbol{v}]||^2_{0,e}\bigg)^{\frac{1}{2}} \nonumber\\
  & & \leq C h^{\frac{d-1}{2}}||\boldsymbol{w}||_{2}|||\boldsymbol{v}|||.
\end{eqnarray}
Using (\ref{50}) and (\ref{59}), we have
\begin{eqnarray}\label{61}
&&-\theta\sum_{e\in\mit\Gamma_{h}\cup\mit\Gamma_D}\int_e\{\boldsymbol{\sigma}
   (\boldsymbol{\varphi})\boldsymbol{n}\}\cdot[\boldsymbol{u}-\boldsymbol{u}_h]d\ell   \nonumber\\
   &&\leq Ch^{\frac{d-1}{2}}||\boldsymbol{\varphi}||_2 |||\boldsymbol{u}-\boldsymbol{u}_h|||.
\end{eqnarray}
Using (\ref{57}), (\ref{59}), (\ref{58}) and (\ref{61}), we have
\begin{eqnarray}\label{62}
   ||\boldsymbol{u}-\boldsymbol{u}_h||_0 &\leq& Ch|||\boldsymbol{u}-\boldsymbol{u}_h|||+
    Ch^{\frac{d-1}{2}}|||\boldsymbol{u}-\boldsymbol{u}_h|||\nonumber\\
   &\leq& Ch\cdot h^{\mu-1}||\boldsymbol{u}||_s +Ch^{\frac{d-1}{2}}\cdot h^{\mu-1}||\boldsymbol{u}||_s\nonumber\\
   &=& C h^{\mu}||\boldsymbol{u}||_s + Ch^{\mu+\frac{d-3}{2}}
   ||\boldsymbol{u}||_s,
\end{eqnarray}
which completes the proof.
\end{proof}

\section{Numerical tests}
\setcounter{equation}{0}

In this section, we present a 2-D numerical example  in $\Omega=(-1, 1)\times(-1, 1)$
with homogeneous Dirichlet boundary condition and empty Neumann boundary.

Let $\lambda=0.03$, $\mu=0.035$ and
\begin{eqnarray*}
  \boldsymbol{f}(x,y)=\lambda\Big(\frac{\pi^2}{4}\zeta_1,\frac{\pi^2}{4}\zeta_1\Big)^T
  +2\mu\Big(\frac{\pi^2}{4}\zeta_2
  +\frac{\pi^2}{8}\zeta_1,
  \frac{\pi^2}{4}\zeta_2+\frac{\pi^2}{8}\zeta_1\Big)^T
\end{eqnarray*}
with $\zeta_1=\cos(\frac{\pi}{2}x+\frac{\pi}{2}y)$ and $\zeta_2=\cos(\frac{\pi}{2}x)\cos(\frac{\pi}{2}y)$.

It is easy to check that the exact solution is
\begin{equation*}
  \boldsymbol{u}(x,y)=\Big(\cos(\frac{\pi}{2}x)\cos(\frac{\pi}{2}y),
  \cos(\frac{\pi}{2}x)\cos(\frac{\pi}{2}y)\Big)^T.
\end{equation*}

In the computation, we set $\beta=125$. For the adjoint inconsistent methods, we use superpenalization  and choose $d=3$. The numerical results of errors in $L^2$-norm and the energy norm are displayed in Table 1 as follows.
\begin{table}[H]
\setlength{\abovecaptionskip}{15pt}   % 0.5cm as an example
\setlength{\belowcaptionskip}{10pt}   % 0.5cm as an example
\small
\caption{Errors in $L^2$-norm and the energy norm}
\begin{tabular}{c c c c c c c c}
  \hline
  Method & k,d & & $h=2^{-1} $& $ 2^{-2}$& $ 2^{-3} $& $2^{-4}$& $2^{-5}$ \\ \hline
  SIPG & $k=1$ & $||\boldsymbol{u}-\boldsymbol{u}_h||_0$ & 0.12213 & 0.03113 & 0.00745 & 0.00150 &0.00038  \\
       &       & $|||\boldsymbol{u}-\boldsymbol{u}_h|||$ & 0.20320 & 0.10402 & 0.05375 & 0.02985 &0.01982  \\

  IIPG & $k=1$ & $||\boldsymbol{u}-\boldsymbol{u}_h||_0$ & 0.12256 & 0.03161 & 0.00796 & 0.00199 &0.00049  \\
       & $d=3$ & $|||\boldsymbol{u}-\boldsymbol{u}_h|||$ & 0.20305 & 0.10333 & 0.05190 & 0.02598 &0.01299  \\

  NIPG & $k=1$ & $||\boldsymbol{u}-\boldsymbol{u}_h||_0$ & 0.12275 & 0.03171 & 0.00799 & 0.00200 &0.00050  \\
       & $d=3$& $|||\boldsymbol{u}-\boldsymbol{u}_h|||$  & 0.20306 & 0.10333 & 0.05190 & 0.02598 &0.01299 \\
  \hline
\end{tabular}
\end{table}

The comparisons of $||\boldsymbol{u}-\boldsymbol{u}_h||_0$, $|||\boldsymbol{u}-\boldsymbol{u}_h|||$
in $\ln$-$\ln$ scale for all three methods are displayed in Figure \ref{fg1} and Figure \ref{fg2}.

\begin{figure}[H]
\centering
\includegraphics[width=0.7\textwidth]{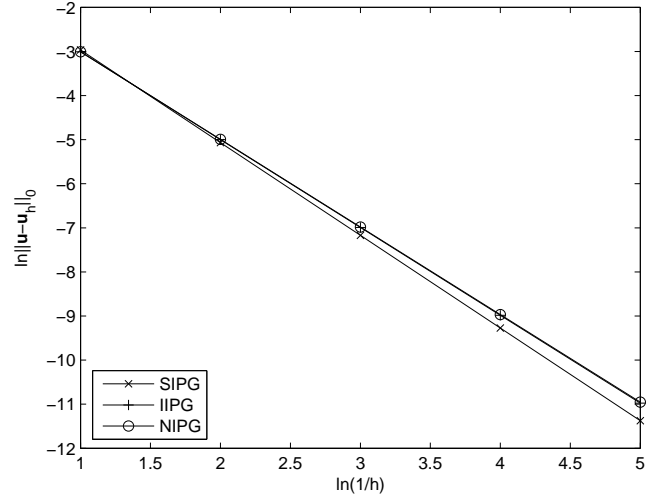}
\caption{$||\boldsymbol{u}-\boldsymbol{u}_h||_0$ in $\ln$-$\ln$ scale for the three methods}\label{fg1}
\end{figure}
\begin{figure}[H]
\centering
\includegraphics[width=0.7\textwidth]{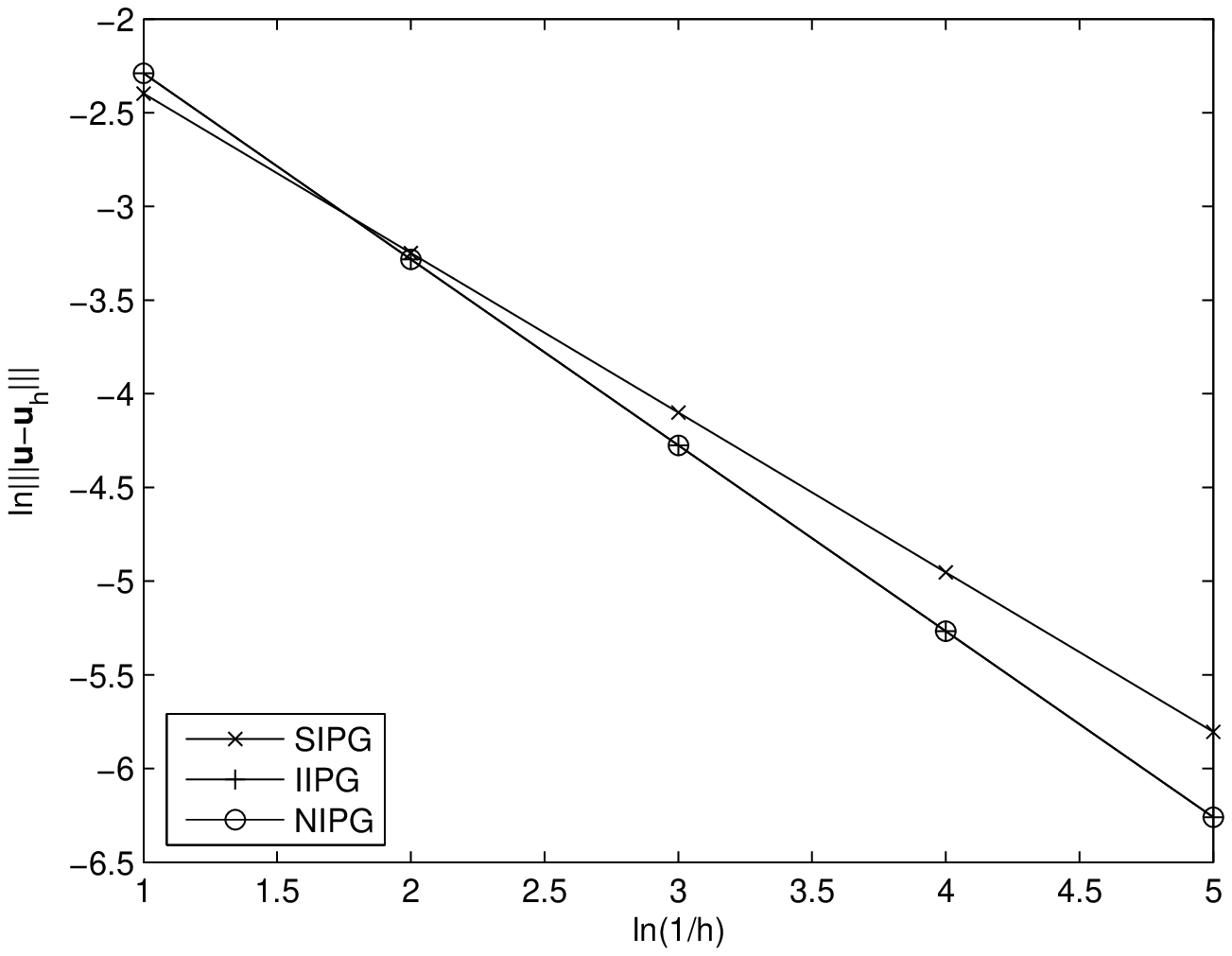}
\caption{$|||\boldsymbol{u}-\boldsymbol{u}_h|||$ in $\ln$-$\ln$ scale for the three methods}\label{fg2}
\end{figure}

From the above figures and Table 1, we find that the optimal convergence rate in the energy norm is got for the three methods, and the optimal convergence rate in $L^2$-norm is achieved for SIPG method, and are obtained for both IIPG method and NIPG method when $d=3$, which conform with the theoretical results of Theorem \ref{th2} and Theorem \ref{th3}.

\end{document}